\documentclass[amstex,12pt]{article}
\usepackage{graphicx}

\usepackage{color}

\newcommand*{\Cdot}{\raisebox{-0.4ex}{\scalebox{1.8}{$\hspace{.3mm}\cdot\hspace{.3mm}$}}}
\newcommand*{\mCdot}{\raisebox{-0.4ex}{\scalebox{1.8}{$\hspace{.3mm}\cdot\hspace{-.2mm}$}}}

\usepackage{amsmath}
\usepackage{amssymb}
\usepackage{amscd}
\usepackage{bm}
\usepackage{latexsym}
\usepackage{makeidx}
\usepackage{tocloft}
\usepackage{fancyhdr}

\usepackage[all]{xy}

\usepackage{ifthen}

\newcommand{\lsection}[2][""]{%
    \ifthenelse{\equal{#1}{""}}{%
        \section{#2}
    }{%
        \renewcommand{\sectionmark}[1]{\markright{\thesection.\ \MakeUppercase{#1}}}
        \section{#2}
        \renewcommand{\sectionmark}[1]{\markright{\thesection.\ \MakeUppercase{##1}}}
    }
}

\newcommand{\lchapter}[2][""]{%
    \ifthenelse{\equal{#1}{""}}{%
        \chapter{#2}
    }{%
        \renewcommand{\chaptermark}[1]{\markboth{\MakeUppercase{\chaptername\ \thechapter.\ #1}}{}}
        \chapter{#2}
        \renewcommand{\chaptermark}[1]{\markboth{\MakeUppercase{\chaptername\ \thechapter.\ ##1}}{}}
    }
}

\def\a{\alpha}

\def\l{\lambda}
\def\iso{isomorphism}
\def\homo{homomorphism}
\def\ra{\rightarrow}

\def\mfd{manifold}
\def\wrt{with respect to}
\def\nbd{neighborhood}
\def\hra{\hookrightarrow}
\def\lra{\longrightarrow}
\def\irr{irreducible}

\def\U{{\cal U}}

\def\W{{\cal W}}

\def\ssm{\hspace{-.5mm}\smallsetminus\hspace{-.5mm}}

\def\o{\omega}

\def\vL{\varLambda}
\def\vP{\varPsi}
\def\vG{\varGamma}
\def\vD{\varDelta}
\def\vv {\vskip.2cm}
\def\st{such that}
\def\op{\operatorname}

\def\C{{\mathbb C}}
\def\R{{\mathbb R}}
\def\Z{{\mathbb Z}}
\def\Q{{\mathbb Q}}

\def\bo1{{\text{\bold 1}}}

\def\r{respectively}

\def\Randmarke#1{\vadjust{\vbox to 0pt
                {\vss \hbox to \hsize{\hskip\hsize
                                     \quad #1\hss}\vskip3.5pt}}}
\def\Randpfeilmarke#1{\Randmarke{$\scriptscriptstyle\Leftarrow$#1}}
\def\bi#1 { $^{\bf!}$ {\{\sl #1 \}} \Randpfeilmarke{\bf !!}}

\newtheorem{theorem}{Theorem}[section]

\newtheorem{corollary}[theorem]{Corollary}
\newtheorem{proposition}[theorem]{Proposition}
\newtheorem{exa}[theorem]{Example}

\newtheorem{exas}[theorem]{Examples}

\newtheorem{prope}[theorem]{Property}

\newtheorem{defini}[theorem]{Definition}
\newenvironment{definition}{\begin{defini} \em}{\end{defini}}

\newtheorem{rema}[theorem]{Remark}
\newenvironment{remark}{\begin{rema} \em}{\end{rema}}

\newenvironment{equationth}{\stepcounter{theorem}\begin{equation}}{\end{equation}}

\newenvironment{proof}{{\noindent \sc Proof: } }{\mbox{ }\hfill$\Box$
                        \vspace{1.5ex} \par}

\def\hfleche#1#2{\smash{\mathop{\vbox{\hbox to 8.5mm{{#1}}}}\limits^{#2}}}

\title{Local and global coincidence homology classes}

\author{Jean-Paul Brasselet\thanks{Partially supported by the FAPESP grant no. 2015/06697-9.}
and Tatsuo Suwa\thanks{Partially supported by the JSPS grants no. 24540060 and no. 16K05116.}}

\date{}

\begin{document}

\pagestyle{plain}

\maketitle \thispagestyle{empty}

{\sl 
\hfill Comme c'est curieux ! 

\hfill Comme c'est bizarre ! 

\hfill et {\bf  quelle co\"\i ncidence} ! 

\hfill J'ai pris le m\^eme train, Monsieur, moi aussi !
}

\hfill {\small Eug\`ene Ionesco}

\hfill {\small La cantatrice chauve.}

\bigskip

\maketitle

\noindent
{\bf Abstract}
\smallskip

For two differentiable maps between two manifolds of possibly different dimensions, the local and
global coincidence homology classes are introduced and studied by Bisi-Bracci-Izawa-Suwa  (2016) in the framework of \v{C}ech-de~Rham cohomology. We take up the problem from the combinatorial viewpoint and give some finer results, in particular for the
local classes. As to the global class, we clarify the relation with  the cohomology coincidence class as studied by Biasi-Libardi-Monis (2015).
In fact they  introduced such a class   in the context of several maps and we also consider this case. In particular we define the local  homology class and
give some explicit expressions. These all together lead to a generalization of the classical Lefschetz coincidence
point formula.

\bigskip


\vv


\begin{section}{Introduction}

After Poincar\'e \cite{Po} and Brouwer \cite{Bro},  the Lefschetz fixed point theorem provided a new insight on the fixed point theory.
S. Lefschetz \cite {L3} proved that, given a self map $f$ of a compact oriented  manifold $M$, the sum of indices at fixed points 
is equal to the alternating sum of the matrix  traces of the linear maps induced by $f$ on the homology of $M$ 
with rational coefficients. 

In fact, Lefschetz \cite{L1} provided the result in the context of coincidences. Given two maps $f$ and $g$ between compact oriented manifolds $M$ and $N$ of the same dimension, 
the coincidence points are defined as points $x\in M$ such that $f(x)=g(x)$. At these points,
one defines the coincidence index and the Lefschetz result expresses the sum of indices in terms of 
alternating sum of suitable matrix  traces (cf. Theorem \ref{lefcpf} below). The Lefschetz fixed point formula is just the case 
$M=N$ and $g$ is the identity map $1_M$ on $M$.

The result has been generalized in two ways.
The first one is the case of manifolds with different dimensions. In fact it seems that Lefschetz himself possibly considered this case (cf. p.\,28 in \cite{St}).
There have been a number of literatures on this  (e.g., \cite{S} and references therein).
In \cite{BBIS}, the global and local homology classes are introduced
and some explicit formulas are given.
The second one is  the case of multi-coincidence. In  \cite{BLM}, the Lefschetz class for several maps is introduced and studied.

In this paper, we will recall the main definitions and results concerning the Lefschetz coincidence indices and classes
in the case of manifolds with same and possibly different dimensions. We study the global 
and local classes of \cite{BBIS} from combinatorial viewpoint and give some finer results. We also study the situation of multicoincidence, providing an alternative definition for the global class in homology. We further define  local classes
and give some explicit formulas. 

The paper is organized as follows. In Section 2 we  review, from the combinatorial viewpoint,  the Poincar\'e and Alexander dualities, the Thom class
and localized intersection products. After a brief description of the coincidence problem in Section 3, we recall in Section 4,
the original Lefschetz coincidence point formula and sketch an outline of the proof. One of the key ingredients is that the local index is given as
the local mapping degree of the difference of the maps. 

We then take up the case of two maps between
\mfd s of possibly different dimensions in Section 5. We recall the definitions of the global and local  coincidence classes 
(Definitions \ref{gch} and \ref{lch}) as well as a general coincidence point theorem
(Theorem \ref{thlefgen}). For the global class we prove that it correspnds
to the cohomology coincidence class of \cite{BLM} via the Poincar\'e duality (Theorem \ref{thcoincoin}). The local class
is representen by a cycle of the form $\sum c_{\bm s}{\bm s}$, where ${\bm s}$ runs through the simplices of an appropriate
dimension in the coincidence set. Thus once we know the coefficient $c_{\bm s}$ for each simplex ${\bm s}$, we have
an explicit expression for the local class. This can be done with the aid of  the representation of the Thom 
 in the \v{C}ech-de~Rham cohomology. A general formula is given in Theorem \ref{classloccdr} and, in some specific 
 cases, it leads to the formulas as in Theorems \ref{propisolated}, \ref{pm} and Corollary \ref{corloccoin}.
 
 Finally we consider the case of several maps in Section 6. Since this case can be reduced to the case of two maps,
 everything done for two maps may be applied to obtain the results for this case (Theorems \ref{thcoinmult}, \ref{locmult}
 and Corollary  \ref{corloccoingen}).
\end{section}

\begin{section}{Basic tools of algebraic topology }\label{secbasics}

 In the sequel, the homology is that of locally finite chains 
and the cohomology is that of cochains on finite
chains, both with integral coefficients, unless otherwise stated.
Also for a cycle $C$, its class in the homology of the ambient space is denoted by $[C]$, while the class in the homology of its support is simply denoted by $C$.

\subsection{Poincar\'e and Alexander dualities}

Combinatorial definitions of the dualities presented here can be found 
in \cite{Br}, see also  \cite{Su1} and \cite{Su2}.

Let $M$ be a connected  oriented $C^\infty$ manifold of dimension $m$. We take a triangulation $K_0$ of $M$
and let $K$ be the barycentric subdivision of $K_0$.  We further take the barycentric subdivision $K'$ of $K$. 
We take the second barycentric subdivision so that the star $S_{K'}(S)$ of a subcomplex $S$ 
\wrt\ 
$K'$, i.e., the union of simplices of $K'$ intersecting with $S$,  proper deformation
retracts to $S$.
Let $K^*$ denote
 the  cell decomposition dual to $K$. Thus for a simplex $\bm{s}$ of $K$ its dual cell $\bm{s}^{*}$  is the union of 
 simplices of $K'$ that intersect with $\bm{s}$ only at the barycenter $b_{\bm{s}}$ of $\bm{s}$. We orient the simplices and cells so that, if $\bm{s}$ is an $(m-p)$-simplex of $K$ and $\bm{s}^*$ its dual
$p$-cell, the orientation of $\bm{s}^*$ followed by that of $\bm{s}$ gives the orientation of $M$. This orientation convention is that of \cite{Su2}.
\paragraph{Poincar\'e duality\,:} 

We  denote by $C^{p}_{K^{*}}(M)$ and $C^{K}_{m-p}(M)$ the groups of $p$-cochains in $K^{*}$ and $(m-p)$-chains
in $K$ and define
\begin{equationth}\label{poinhomo}
P:C^{p}_{K^{*}}(M)\lra C^{K}_{m-p}(M)\qquad\text{by}\ \ 
u\mapsto\sum_{\bm{s}}\langle \bm{s}^*,u\rangle \bm{s},
\end{equationth}
where $\langle\ ,\ \rangle$ denotes the Kronecker paring and the sum is taken over all $(m-p)$-simplices $\bm{s}$ of $K$.
Then it induces the Poincar\'e duality isomorphism
\[
P_M : H^{p}(M) \stackrel{\sim}\lra H_{m-p}(M).
\]
It is shown that $P_{M}$ is given by the  left cap product with the fundamental class of $M$, the class of the sum of all $m$-simplices in $M$. Thus a class $c$ in $H^{p}(M)$ corresponds to the class $\gamma$ in $H_{m-p}(M)$ \st
\begin{equationth}\label{pdual}
\langle M, c\smallsmile a\rangle=\langle \gamma,a\rangle\qquad\text{for all} \ a\in  H^{m-p}(M).
\end{equationth}

In the sequel $P_M$ will simply be denoted by $P$ if there is no fear of confusion.

\begin{remark}\label{remcap} In some literature such as \cite{V}, the  Poincar\'e duality \iso\ is defined so that it is given by the right
cap product with the fundamental class. If we denote this \iso\ by $P_{M}'$, we have
\[
P_{M}'=(-1)^{p(m-p)}P_{M}.
\]
\end{remark}

\paragraph{Alexander duality\,:}

Let $S$ be a $K_0$-subcomplex of $M$. We denote by $S_{K'}(S)$ and $O_{K'}(S)$ the star and the open star of $S$ in $K'$ and set
\[
C^{p}_{K^{*}}(M,M\ssm O_{K'}(S))=\{\,u\in C^{p}_{K^{*}}(M)\mid\langle\bm{s}^{*},u\rangle=0\ \text{for}\ \bm{s}\not\subset S\,\}
\]
Then (\ref{poinhomo}) induces
\[
C^{p}_{K^{*}}(M,M\ssm O_{K'}(S))\lra C^{K}_{m-p}(S)
\]
and the Alexander isomorphism 
\[
A_{M,S} : H^{p}(M, M \ssm S) \stackrel{\sim}\lra H_{m-p} (S).
\]
In the sequel $A_{M,S}$ is denoted by $A$ if there is no fear of confusion.
We have  the following commutative diagram:
\begin{equationth}\label{PoinAlex}
\SelectTips{cm}{}
\xymatrix 
{H^{p}(M, M \ssm S) \ar[r]^-{j^{*}} \ar[d]^{A}_-{\wr}  &  H^{p}(M)\ar[d]^{P}_-{\wr}\\
H_{m-p}( S) \ar[r]^-{i_{*}} & H_{m-p} (M),
}
\end{equationth}
where $i: S \to M$ and $j : (M, \emptyset) \to (M,S)$ are inclusions. 

\paragraph{Thom \homo\,:} Let $X$ be an oriented pseudo-manifold of dimension $d$  in $M$ (cf. Definition \ref{defpmfd} below). For a $(d-p)$-simplex
$\bm{s}$ of $K$, $\bm{s}^{*}\cap X$ is a $p$-chain of $K'$. The \homo
\[
P:C^{p}_{K'}(X)\lra C^{K}_{d-p}(X)\qquad\text{ given by}\ \ 
u\mapsto\sum_{\bm{s}}\langle \bm{s}^*\cap X,u\rangle \bm{s},
\]
 induces the Poincar\'e \homo
\[
P_X : H^{p}(X) \lra H_{d-p}(X).
\]
It is also given by the cap product with the fundamental class of $X$.

Setting $k=m-d$, the Thom \homo
\[
T_{M,X}:H^{p}(X)\lra H^{p+k}(M,M\ssm X)
\]
is induced from the \homo
\[
T:C^{p}_{K'}(X)\lra C^{p+k}_{K^{*}}(M,M\ssm O_{K'}(X))\qquad\text{ given by}\ \ 
\langle \bm{s}^*,T(u)\rangle=\langle \bm{s}^*\cap X,u\rangle.
\]

Extending (\ref{PoinAlex}) for $S=X$, we have  
the following commutative pentahedron (square pyramid) where $P_X$
and  $T_{M,X}$ are isomorphisms if $X$ is a manifold. 
The map $i_! $ is the Gysin map
and the map $a$ is defined as $a=P_M \circ j^* =i_* \circ A_{M,X}$. 

\begin{equationth} \label{pentahedron}
 \SelectTips{cm}{}
\xymatrix @C=0.5cm 
{ 
&&&H^{p+k}(M, M \ssm X)  \ar[llldd]_-{j^*} \ar@{-->}[rdd]_{a} \ar[rrrddd]^{A_{M,X\;  \simeq}}\\
&&&&&\\
H^{p+k} (M)  \ar@{-->}[rrrr]_-{P_M \; \simeq}  &&&& H_{d-p}( M)  &\\
&H^{p}(X)  \ar[rrrrr]_-{P_X  \; (\simeq)} \ar[lu]^{i_!}   \ar[rruuu]^-{T_{M,X} \; (\simeq)}&&&&&H_{d-p}( X).  \ar@{-->}[llu]^{i_*}&\\}
\end{equationth}

\begin{definition} The {\em Thom class} $\vP_{X}$ of $X$ is a class in $H^{k}(M, M \ssm X) $ defined as
\[
\vP_{X}=T_{M,X}([1]),\qquad [1]\in H^{0}(X).
\]
\end{definition}

Thus it may also be written as 
\begin{equationth} \label{thomalex}
\vP_{X}=(A_{M,X})^{-1}X.
\end{equationth}

\begin{remark}\label{remthom} If we use the convention that the duality \homo s are given by the right cap product (cf. Remark \ref{remcap}), then the Thom class $\vP_{X}'$ in this convention is given by
\[
\vP_{X}'=(-1)^{d(m-d)}\vP_{X}.
\]
\end{remark}

\subsection{Intersection product and localized intersection product}

Let $M$ be a connected  oriented $C^\infty$ manifold of dimension $m$, as before. For two classes $a\in H_{r}(M)$ and $b\in H_{s}(M)$, 
the {\em intersection product} $a\Cdot b$ is 
defined by
\begin{equationth}\label{defint}
a\Cdot b=P(P^{-1}a\smallsmile P^{-1}b)\qquad\text{in}\ \  H_{r+s-m}(M),
\end{equationth}
where $\smallsmile$ denotes the cup product. Then $a\Cdot b$ is additive in $a$ and $b$ and we have
\begin{equationth}\label{anticomm}
b\Cdot a=(-1)^{(m-r)(m-s)}a\Cdot b.
\end{equationth}

Suppose $M$ is compact.  In this case, if $r+s=m$, then $a\Cdot b$ is in $H_{0}(M)=\Z$ and may be thought of as an integer.

\begin{remark} The above intersection product remains the same if we use $P_{M}'$ instead of $P_{M}$.
\end{remark}

Let $S_1$ and $S_2$ be two $K_0$-subcomplexes of $M$ and set $S=S_1\cap S_2$. Let $A_1$, $A_2$ and $A$ denote the Alexander \iso s for $(M,S_1)$, $(M,S_2)$ and $(M,S)$, \r.
For two classes $a\in  H_{r}(S_1)$ and $b\in  H_{s}(S_2)$, the {\em  intersection
product $(a\Cdot b)_S$ localized at $S$} is defined by
\begin{equationth}\label{locint}
(a\Cdot b)_S=A(A_1^{-1}a\smallsmile A_2^{-1}b)\qquad\text{in}\ \  H_{r+s-m}(S),
\end{equationth}
where $\smallsmile$ denotes the cup product
\[
H^{m-r}(M,M\ssm S_1)\times H^{m-s}(M,M\ssm S_2)\stackrel\smallsmile\lra H^{2m-r-s}(M,M\ssm S).
\]
Then $(a\Cdot b)_S$ is additive in $a$ and $b$ and satisfies a relation similar to (\ref{anticomm}). Letting $i_1:S_1\hra M$, $i_2:S_2\hra M$ and $i:S\hra M$ be the inclusions, from 
(\ref{PoinAlex}) we see that the definitions (\ref{defint}) and (\ref{locint}) are consistent in the sense that
\[
(i_1)_*a\Cdot (i_2)_*b=i_*(a\Cdot b)_S.
\]

In the above situation, suppose $S$  has a finite number of connected components $(S_{\l})$. Then
$H_{r+s-m}(S)=\bigoplus_{\l} H_{r+s-m}(S_{\l})$  and we have the intersection product $(a\Cdot b)_{S_{\l}}$ in
$H_{r+s-m}(S_{\l})$ for each $\l$. We have the ``localization formula''
\[
(i_1)_*a\Cdot (i_2)_*b=\sum_{\l}(i_{\l})_{*}(a\Cdot b)_{S_{\l}}\qquad\text{in}\ \ H_{r+s-m}(M),
\]
where $i_{\l}:S_{\l}\hra M$ denotes the inclusion. It is a formula of the same nature as the one given in Theorem
\ref{thresidue} below.
Suppose $S$ is compact. Then, if $r+s=m$, $H_{0}(S_{\l})=\Z$ for each $\l$ and $(a\Cdot b)_{S_{\l}}$
is an integer. In this case, the above formula is written as
\begin{equationth}\label{localizintno}
(i_1)_*a\Cdot (i_2)_*b=\sum_{\l}(a\Cdot b)_{S_{\l}}.
\end{equationth}

The above dualities and intersection products may also be defined in the framework of 
\v{C}ech-de~Rham cohomology (cf. \cite{BBIS}, \cite{Su}, \cite{Su2}).

\end{section}

\begin{section}{Coincidence problem}

Let $M$  and $N$ denote    connected  oriented $C^\infty$  manifolds of  dimensions $m$ and $n$, \r. 
Let us consider two $C^\infty$ maps $f: M \to N$ and $g:M \to N$.

\begin{definition}
The coincidence set ${\rm Coin}(f,g)$ is defined as
\[ 
{\rm Coin}(f,g) = \{\, x\in M\mid f(x)=g(x)\, \}.
\]
\end{definition}

In the subsequent sections we introduce global and local invariants for the pair $(f,g)$, study the relation among them and try to express the invariants explicitly. We then generalize the results to the case of several maps.

We denote by $\vG_{f}$ and $\vG_{g}$ the graphs of $f$ and $g$
in $M\times N$, \r. They are both $m$-dimensional sub\mfd s of $M\times N$. We orient $\vG_{f}$ so that the map $\tilde f:M\ra \vG_{f}\subset M\times N$ given by $\tilde f(x)=(x,f(x))$ is an orientation preserving diffeomorphism, and similarly for $\vG_{g}$.
\end{section}

\begin{section}{Lefschetz coincidence point formula - case of compact manifolds of the same dimension}\label{secsamedim}
 
In this section we review the original Lefschetz coincidence point formula in the case of $C^\infty$ maps
between manifolds of the same dimension,
which will be generalized in various settings in the subsequent sections. 

Let $M$  and $N$ denote  compact  connected  oriented $C^\infty$  manifolds of the same dimension $m$. 
Let us consider two maps $f: M \to N$ and $g:M \to N$.

\paragraph{Lefschetz  number\,:}  In this paragraph, we consider homology and cohomology with $\Q$
coefficients and denote by $f_{p}$ and $f^{p}$  the \homo s induced by $f$ on the $p$-th homology and $p$-th cohomology,
\r, and similarly for $g_{p}$ and $g^{p}$.
For each $p$, we have the commutative  diagram\,:
\[
\SelectTips{cm}{}
\xymatrix 
{ 
H^{p}(M, \Q)   \ar[r]_-{P_{M}}^-{\sim}& H_{m-p}(M, \Q)\ar[r]^-{\sim}_-{K_{M}}  \ar[d]^-{f_{m-p}}& H^{m-p}(M,\Q)^{*}
\ar[d]^{(f^{m-p})^{*}}\\
H^{p}(N, \Q)\ar[u]^{f^{p}} \ar[r]_-{P_{N}}^-{\sim} &  H_{m-p}(N, \Q)\ar[r]^-{\sim}_-{K_{N}}  &  H^{m-p}(N,\Q)^{*},
 }
 \]
where $K_{M}$ and $K_{N}$ denote  the \iso s  induced by the Kronecker pairing. 
Considering a similar diagram for $g$,
we set
\[
\tilde g^{p}=P_{N}^{-1}\circ g_{m-p}\circ P_{M}:H^{p}(M,\Q)\lra H^{p}(N,\Q).
\]

\begin{definition}The {\em  Lefschetz  number} $\op{Lef}(f,g)$ of the pair $(f,g)$ is  defined by
\[
\op{Lef}(f,g)=\sum_{p=0}^m (-1)^p \cdot \op{tr} (f^{p} \circ \tilde g^{p}).
\]
\end{definition}

We give a proof of the following for the sake of completeness\,:
\begin{proposition} We have
\[
\op{Lef}(g,f)=(-1)^{m}\op{Lef}(f,g).
\]
\end{proposition}
\begin{proof} 
First note that if we set
\[
D_{M}^{p}=(K_{M}\circ P_{M})^{p}:H^{p}(M,\Q)\overset\sim\lra H^{m-p}(M,\Q)^{*},
\]
then its transpose (dual map) is equal to $D_{M}^{m-p}$, i.e., $(D_{M}^{p})^{*}=D_{M}^{m-p}$. We compute,
omitting the superscripts for $D$,
\[
\begin{aligned}
(g^{p}\circ\tilde f^{p})^{*}&=(g^{p}\circ (D_{N}^{-1}) \circ (f^{m-p})^{*}\circ D_{M})^{*}=
D_{M}\circ f^{m-p}\circ  (D_{N}^{-1})\circ (g^{p})^{*}\\
&=D_{M}\circ f^{m-p}\circ  (D_{N}^{-1})\circ (g^{p})^{*}\circ D_{M}\circ D_{M}^{-1}
=D_{M}\circ f^{m-p}\circ  \tilde g^{m-p}\circ D_{M}^{-1}.
\end{aligned}
\]
Thus
\[
\op{tr}(g^{p}\circ\tilde f^{p})=\op{tr}(f^{m-p}\circ\tilde g^{m-p}),
\]
which proves the proposition.
\end{proof}

\begin{remark} {\bf 1.} If we set 
\[
\tilde g_{p}=P_{M}\circ g^{m-p}\circ P_{N}^{-1}:H_{p}(N,\Q)\lra H_{p}(M,\Q),
\]
then we may write 
\[
\op{Lef}(f,g)=\sum_{p=0}^m (-1)^p \cdot \op{tr} (\tilde g_{p} \circ f_{p} ).
\]
\vv

\noindent
{\bf 2.} If we use $P_{M}'$ and $P_{N}'$ (cf. Remark \ref{remcap}) the above remains the same.
\end{remark}
    
\paragraph{Local mapping degree\,:} Let $f,g:M\ra N$ be as above.
Suppose $p$ is an isolated point in $\op{Coin}(f,g)$.
Let $U$ be a coordinate \nbd\ around $p$ with  
coordinates $x=(x_1,\dots,x_m)$ in $M$ and $V$ a coordinate \nbd\ around $f(p)=g(p)$
in $N$. Also let $B$ be a closed ball around $p$ in $U$ \st\ $f(B)\subset V$ and $g(B)\subset V$.
Thus we may consider the map $g-f:B\ra\R^m$ whose image is the origin $0$ in $\R^m$ only at $p$. The boundary $\partial B$
is homeomorphic to the unit sphere $S^{m-1}$ and we have the map
\[
\gamma:\partial B\lra S^{m-1}\qquad\text{defined by}\ \ \gamma(x)=\frac{g(x)-f(x)}{\Vert g(x)-f(x)\Vert}.
\]
We denote the degree of this map by $\op{deg}(g-f,p)$ and call it the {\em coincidence index} of $(f,g)$ at $p$.

An isolated coincidence point $p$ of the pair $(f,g)$ is said to be {\em non-degenerate} if
\[
\det(J_g(p)-J_f(p))\ne 0,
\]
where  $J_f(p)$ and 
$J_g(p)$ denote the Jacobian matrices of $f$ and $g$ at $p$. 
 If $p$ is a non-degenerate coincidence point, then we have
\[
\op{deg}(g-f,p)=\op{sgn}\, \det(J_g(p)-J_f(p)).
\]

With these, we have\,:
\begin{theorem}[Lefschetz coincidence point formula] Let $M$ and $N$ be compact connected oriented $C^{\infty}$ 
manifolds of the same dimension
and let $f , g \colon M \ra N$ be $C^\infty$ maps.
Suppose $\op{Coin}(f,g)$ consists   of a finite number of isolated points. Then
\[
\op{Lef}(f,g)=\sum_{p \in \op{Coin}(f,g) }\deg(g-f,p).
\]
In particular, $\op{Lef}(f,g)$ is an integer, which is zero if $\op{Coin}(f,g)=\emptyset$.
\label{lefcpf}
\end{theorem}

The above theorem applied to the case $N=M$ and $g=1_M$, the identity map of $M$, gives  the Lefschetz fixed point formula for $f$.

Basically the proof of the above theorem consists of the following two parts.
Note that, in the case under consideration, $\vG_{f}$ and $\vG_{g}$ are both $m$-cycles and
the intersection product $[\vG_{f}]\Cdot [\vG_{g}]$ is in $H_{0}(M\times N)=\Z$.
 \vv
    
\noindent
{\bf Part I.} To show that $\op{Lef}(f,g)=[\vG_{f}]\Cdot [\vG_{g}]$.
\vv
    
\noindent
{\bf Part II.} To show that $[\vG_{f}]\Cdot [\vG_{g}]$ localizes at the points of $\op{Coin}(f,g)$ and the local contribution from each point is equal to the local mapping degree.
\vv

In some literature such as \cite{V}, the intersection number does not appear explicitly. Instead the following number is
introduced.
Let $\vD_{N}$ denotes the diagonal in $N\times N$ and define $(f,g):M\ra N\times N$  by $(f,g)(x)=(f(x),g(x))$.
Consider the commutative diagram\,:
\[
\SelectTips{cm}{}
\xymatrix 
{ 
H^m(N \times N, N\times N \ssm \vD_N) \ar[r]^-{j^*}  \ar[d]^{A'_{N\times N,\vD_{N}}}_{\wr} &
H^m(N \times N) \ar[r]^-{(f,g)^*}  \ar[d]^{P'_{N\times N}}_{\wr}
& H^m(M)  \ar[d]^{P'_M}_{\wr} \\
H_m(\vD_N) \ar[r]^-{i_*} & H_m(N\times N) & H_{0}(M), \\
}
\]

The {\em Lefschetz class} of $(f,g)$ is defined by (cf. Remarks \ref{remcap} and \ref{remthom})

\begin{equationth}\label{eqC4}
L(f,g) =  (f,g)^* \circ j^* (\vP'_{\vD_{N}}) =(f,g)^* \circ (P_{N\times N}')^{-1} ([\vD_N] )\qquad\text{in}\ \ H^{m}(M).
\end{equationth}

Then we have (cf. Theorem  \ref{thcoincoin} below)

\[
P'_{M}L(f,g)=[\vG_{f}]\Cdot [\vG_{g}].
\]

Part I is of purely global nature and  can be done by direct computations taking bases of homology or cohomology of $M$ and $N$ (\cite{V}, also \cite{BBIS}). For this, it is not necessary to consider the Thom class $\vP'_{\vD_{N}}$.
In \cite{V} it is used for Part II. See also \cite{GHV} for a similar approach and the statement of the theorem as above.
In \cite{BBIS}, a simple direct proof for these is given using the expression of the Thom class of $\vG_{g}$ in the framework of
the \v{C}ech-de~Rham cohomology.
\end{section}

\section{Coincidence of two maps between manifolds of possibly different dimensions}\label{sectwomapsdd}

In this section, we generalize the results in the previous section to the case of  two maps between manifolds with possibly different dimensions.

\subsection{Intersection product with a map} 

We  recall the notion of intersection product with a map, see \cite{BBIS}, \cite{Su1} and \cite{Su2}, also \cite{Fu} in the 
algebraic category.

\begin{definition}
Let $W$ and $M$ be oriented $C^\infty$ manifolds of dimensions $m'$ and $m$, respectively, and 
$F : M \to W$ a $C^\infty$ map. We define the intersection product $M\mCdot_F\  $ so that the first diagram below is 
commutative. 
Also, for a  subcomplex  $\tilde S$ of a triangulation of  $W$, we set $S = F^{-1}({\tilde S})$ and suppose 
$S$ is  a subcomplex of a triangulation of $M$. 
We then define the localized intersection product $(M\mCdot_F\ )_S$ so that the second diagram is commutative\,:
\[
\SelectTips{cm}{}
\xymatrix
{H^p(W) \ar[r]^-{\sim}_-{P} \ar[d]^{F^*} & H_{m'-p}( W) \ar[d]^{ M\cdot_{F} } \\
H^{p}(M) \ar[r]^-{\sim}_-{P} & H_{m-p} (M),}\hspace{1.5cm}
\xymatrix{H^p(W,W\ssm {\tilde S}) \ar[r]^-{\sim}_-{A} \ar[d]^{F^*} & H_{m'-p}( \tilde S) \ar[d]^{ (M \cdot_{F}\ )_S} \\
H^{p}(M, M\ssm S) \ar[r]^-{\sim}_-{A} & H_{m-p} (S).}
\]
\end{definition}

Note that, if $M$ is a sub\mfd\ of $W$ and if $F=\iota:M\hra W$ is the inclusion, the products $M\mCdot_\iota\ $
and $(M\mCdot_{\iota}\ )_S$ coincide
with $(M\Cdot\ )_M$ and   $(M\Cdot\ )_S$, \r, defined in Section \ref{secbasics} (cf. \cite{BBIS} Proposition 3.9, also \cite{Su1} Section 7).

In the above situation, suppose $S$ has only a finite number of connected components $(S_{\l})_{\l}$. Then  
\[
H_{m-p}(S)=\bigoplus_{\l}H_{m-p}(S_{\l})
\]
 and, for a class $c$ in $H_{m-p}(S)$,
the class $(M\mCdot_{F}\, c)_{S}$ determines a class $(M\mCdot_{F}\, c)_{S_{\l}}$ in $\breve H_{m-p}(S_{\l})$ for each $\l$.
We have the ``residue theorem'', which basically follows from (\ref{PoinAlex})\,:

\begin{theorem}\label{thresidue}
Let $C$ be an $(m'-p)$-cycle in $W$ with support $\tilde S=|C|$. Suppose $S=F^{-1} \tilde S$ has a finite number
of connected components $(S_\l)_{\l}$. Then we have
\[
M\mCdot_{F}\, [C]=\sum_\l(i_\l)_*(M\mCdot_{F} \,C)_{S_\l}\qquad\text{in}\ \ H_{m-p}(M),
\]
where $i_{\l}:S_{\l}\hra M$ denotes the inclusion.
\end{theorem}

\subsection{Global classes}

\subsubsection{The Lefschetz coincidence cohomology class.}

We recall the Lefschetz coincidence cohomology class as considered in \cite{BLM}, which is a generalization of
(\ref{eqC4}). 
See also \cite{S} for related problems.

Let $X$ be a topological space and $N$ an 
oriented manifold of dimension $n$. Suppose we have two maps
 $f, g :X \to N$.
Then we have the map $(f,g):X\to N\times N$ and the  diagram\,:
\[
\SelectTips{cm}{}
\xymatrix { 
H^n(N\times N, N\times N \ssm \vD_N)  \ar[r]^-{j^*} \ar[d]_-{\wr}^-{A'_{N\times N,\vD_N}}&
H^n(N\times N)  \ar[r]^-{(f,g)^*}  \ar[d]_-{\wr}^-{P'_{N\times N}}   & H^n(X)\\
H_n ( \vD_N)  \ar[r]^-{i_*}  &H_n( N\times N).  &\\
}
\]
Denoting by $\vP'_{\vD_{N}}=(A')^{-1}\vD_{N}$ the Thom class of
$\vD_{N}$, 
the {\em Lefschetz coincidence cohomology class} is defined 
as 
\begin{equationth}\label{Lefcoh2}
L(f,g) = ( j \circ (f,g))^* (\vP'_{\vD_N}) =(f,g)^* (P')^{-1} [\vD_N]=(-1)^{n}(f,g)^* P^{-1} [\vD_N]\qquad\text{in}\ \ H^n(X).
\end{equationth}
In \cite{BLM} it is assumed that $N$ is compact and the cohomology is with rational coefficients.
Note that $L(g,f) =(-1)^{n^{2}}L(f,g) =(-1)^{n}L(f,g)$.
Note also that as long as we consider the global class, it is not necessary to consider the Thom class of $\vD_{N}$.

In the case $X=M$ is an oriented \mfd\ of dimension $m$, 
we have the commutative diagram 
\[
\SelectTips{cm}{}
\xymatrix 
{ 
H^n(N\times N)  \ar[d]^-{P_{N\times N}}_-{\wr} \ar[r]^-{(f,g)^*}  & H^n(M) \ar[d]_-{\wr}^-{P_{M}}\\
H_n( N\times N)\ar[r]^-{M \cdot_{(f,g)}} & H_{m-n} (M)}
\]
so that
\begin{equationth}\label{Lefcoh}
P_ML(f,g) =(-1)^{n}M\mCdot_{(f,g)}\hspace{.3mm}[\vD_N].
\end{equationth}

\subsubsection{Global coincidence homology class}

The global homology  coincidence class is defined in \cite{BBIS} in the framework of  \v Cech-de~Rham cohomology.
It can be done as well in the  combinatorial context.

Let $M$ and $N$ be oriented \mfd s of dimensions $m$ and $n$, \r. Suppose we have two maps $f,g:M\ra N$.
Let $1_M$ denote the identity map of $M$ and consider the map $\tilde f = (1_M, f) : M \to M\times N$.
Setting $W=M\times N$,  we have the commutative diagram\,:

\[
\SelectTips{cm}{}
\xymatrix 
{H^n(W) \ar[r]^-{\sim}_-{P_{W}} \ar[d]^{\tilde f^*}
& H_{m}( W) \ar[d]^{M \cdot_{\tilde f} \ }\\
H^{n}(M) \ar[r]^-{\sim}_-{P_{M}} & H_{m-n} (M).}
\]
Let us denote by $\vG_g$ the graph of $g$ in $W$, that defines an $m$-cycle whose 
homology class in $H_m(W)$ is denoted by $[\vG_g]$. 

\begin{definition}\label{gch} The {\em global coincidence homology class} $\vL(f,g)$ is 
defined by
\[
\vL(f,g) =M \mCdot_{\tilde f} \hspace{.3mm}[\vG_g] \qquad\text{in}\ \ H_{m-n}(M).
\]
\end{definition}

Note that $\tilde f$ induces an \iso\ $\tilde f_*: H_{m-n}(M)\stackrel\sim\ra  H_{m-n}(\vG_f)$
and   $\vL(f,g)$ corresponds to $\vG_f\Cdot [ \vG_g]$ in $H_{m-n}(\vG_f)$, which is sent to
$[\vG_f]\Cdot [ \vG_g]$ in $H_{m-n}(W)$ by the canonical \homo\ 
$ H_{m-n}(\vG_f)\ra H_{m-n}(W)$. 

\begin{remark} {\bf 1.} $\vL(g,f)=(-1)^{n^{2}}\vL(f,g)=(-1)^{n}\vL(f,g)$.
\vv

\noindent
{\bf 2.}  In particular, if $m=n$ and if $M$ and $N$ are compact, $\vL(f,g) =[\vG_f]\Cdot [ \vG_g]$
is the intersection number, which in turn coincides with the Lefschetz number $\op{Lef}(f,g)$ defined in Section \ref{secsamedim}.
\end{remark} 

\subsubsection{``Coincidence" of the two Lefschetz coincidence  classes}

\begin{theorem}\label{thcoincoin} If $M$ and $N$ are compact,
\[
P_{M}L(f,g)=(-1)^{n(m-n)}\vL(f,g),\quad\text{i.e.,}\quad P'_{M}L(f,g)=\vL(f,g)\quad\text{in}\ \ H_{m-n}(M,\Q).
\]
\end{theorem}
\begin{proof}
Let us consider the Poincar\'e \iso s
\[
P_{N\times N}:H^{n}(N\times N)\overset\sim\lra H_{n}(N\times N)\quad\text{and}\quad P_{M\times N}:H^{n}(M\times N)\overset\sim\lra H_{m}(N\times N)
\]
and set $\eta_{\vD}=P_{N\times N}^{-1}(\vD_{N})$ and $\eta_{g}=P_{M\times N}^{-1}(\vG_{g})$. By (\ref{Lefcoh}), it suffices to prove
\begin{equationth}\label{eqtoprove}
(g,f)^{*}\eta_{\vD}=(-1)^{n(m-n)}\tilde f^{*}\eta_{g}\qquad\text{in}\ \ H^{n}(M,\Q).
\end{equationth}

Let $\{\mu_{i}^{p} \}_i$ be a  basis of $H^p(M)$.
We set $q=m-p$ and let $\{\check{\mu}^{q}_{j} \}_j$ be the basis of $H^q(M)$ dual 
 to $\{\mu_{i}^{p}\}_i$\,:
\[
\langle M,\mu_{i}^{p}\smallsmile \check{\mu}_{j}^{q}\rangle=\delta_{ij}.
\]
We also take a basis   $\{\nu_k^{p} \}_k$  of $H^p(N)$
and  the  basis $\{\check{\nu}^{r}_\ell \}_\ell$ of $H^q(N)$ dual  to  
$\{\nu_k^{p} \}_k$, $r=n-p$.
By the K\"unneth formula, a basis of $H^n(M\times N)$  is given by
\[
\left\{\,  \xi^{p,r}_{i,\ell}=\pi_1^* \mu_{i}^{p}\smallsmile \pi_2^* \check\nu_\ell^{r}\, \right\}_{p+r=n},
\]
where $\pi_1$ and $\pi_2$ are projections onto the first and second factors. Also a basis of the cohomology $H^n(N\times N)$  is given by
\[
\left\{\,  \eta^{p,r}_{k,\ell}=\varpi_1^* \nu_{k}^{p}\smallsmile \varpi_2^* \check\nu_\ell^{r}\, \right\}_{p+r=n},
\]
where $\varpi_1$ and $\varpi_2$ are projections onto the first and second factors. By definition
\[
\langle N\times N,\eta_{\vD}\smallsmile\varphi\rangle=\langle\vD,\varphi\rangle.
\]
We write 
\[
\eta_{\vD}=\sum a^{p}_{k,\ell}\eta^{p,r}_{k,\ell}
\]
Letting $\varphi=\varpi_1^* \check\nu_{k'}^{p'}\smallsmile \varpi_2^* \nu_{\ell'}^{n-p'}$, we see that 
$a^{p}_{k,\ell}=(-1)^{n-p}\delta_{k,\ell}$ and $\eta_{\vD}=\sum(-1)^{n-p} \eta^{p,r}_{k,k}$. Thus we have
\[
(g,f)^{*}\eta_{\vD}=\sum(-1)^{n-p} g^* \nu_{k}^{p}\smallsmile f^* \check\nu_k^{n-p}.
\]
Next, by definition
\[
\langle M\times N,\eta_{g}\smallsmile\psi\rangle=\langle\vG_{g},\psi\rangle.
\]

We write 
\[
\eta_{g}=\sum b^{p}_{i,\ell}\xi^{p,r}_{i,\ell}
\]
By similar computations, letting $\psi=\pi_1^* \check\mu_{i'}^{p'}\smallsmile \pi_2^* \nu_{\ell'}^{m-p'}$,  we have
\[
\tilde f^{*}\eta_{g}=\sum(-1)^{\varepsilon} g^* \nu_{k}^{p}\smallsmile f^* \check\nu_k^{n-p},
\]
where $\varepsilon=(m-p)p-(n-p)m$, and we have (\ref{eqtoprove}).
\end{proof}

\subsection{Local classes}

\subsubsection{The local coincidence homology class}

We recall the local coincidence homology class defined in \cite{BBIS} in our context.

Note that $\op{Coin}(f,g)=\tilde f^{-1}(\vG_g)$, which will be simply denoted by $C$. 
We have the commutative diagram\,:
\[
\SelectTips{cm}{}
\xymatrix  
{ H^n(W, W\ssm\vG_g) \ar[r]^-{\sim}_-{A} \ar[d]^{\tilde f^*}
& H_{m}(\vG_g) \ar[d]^{(M \cdot_{\tilde f}\ )_C}\\
H^{n}(M,M\ssm C) \ar[r]^-{\sim}_-{A} & H_{m-n} (C).}
\]

\begin{definition}\label{lch} (\cite{BBIS}, Definition 4.2) The {\em local coincidence class} $\vL (f, g; C)$ 
of the pair $(f, g)$ at $C$ is defined to be the localized intersection class :
\[
\vL(f,g;C)= (M\mCdot_{\tilde f} \hspace{.7mm}\vG_g)_C \qquad \text{  in} \ \  H_{m-n}(C).
\]
\end{definition}

If we denote by $\vP_{\vG_g}$ the Thom class of $\vG_g$, we have
\begin{equationth}\label{coinlocthom}
\vL(f,g;C)=A\tilde f^*\vP_{\vG_g}.
\end{equationth}

Suppose $C = \op{Coin}(f,g)$ has a finite number of connected components $(C_\l)_\l$. Then we have 
$H_{m-n}(C) = \oplus_\l H_{m-n}(C_\l)$ and accordingly we have the local coincidence class 
$\vL(f,g;C_\l)$  in $H_{m-n}(C_\l)$. We have a general coincidence point theorem\,:

\begin {theorem}[\cite{BBIS} Theorem 4.4]
In the above situation
\[
\vL(f,g) = \sum_\l (\iota_\l)_* \vL(f,g;C_\l) \qquad \text{ in}\ \  H_{m-n}(M),
\]
where $\iota_\l:C_\l\hra M$ denotes the inclusion.\label{thlefgen}
\end{theorem}

In the sequel we give explicit expressions of the local classes. For this, we use (\ref{coinlocthom}) and a representation 
of the Thom class in the relative \v Cech-de~Rham cohomology.
\subsubsection{\v Cech-de~Rham cohomology}\label{CdR}

Let $M$ be a $C^\infty$ \mfd\ of dimension $m$. For an open set $U$ of  $M$, we
denote by $A^p(U)$  the vector space of complex valued $C^{\infty}$ $p$-forms on $U$. The 
cohomology of the complex $(A^*(M),d)$ is the de~Rham cohomology
$H^*_d(M)$.
The \v Cech-de~Rham cohomology may be defined for an arbitrary open covering of
$M$, however here we only consider  coverings  consisting of  two open sets.

\paragraph{\v Cech-de~Rham cohomology\,:}
Let $\U=\{U_0, U_1\}$ be an open covering of $M$. We set
$U_{01}=U_0\cap U_1$ and define the complex vector space $A^p(\U)$ as
\[
A^p(\U)=A^p(U_0)\oplus A^p(U_1)\oplus A^{p-1}(U_{01}).
\]
An element $\sigma$ in $A^p(\U)$ is given by a triple
$\sigma=(\sigma_0, \sigma_1, \sigma_{01})$ with $\sigma_i$ a $p$-form on $U_i$, $i=0,1$, and
$\sigma_{01}$  a $(p-1)$-form on $U_{01}$. We define an operator $D: A^p(\U) \to A^{p+1}(\U)$ by
\[
D\sigma=(d\sigma_0, d\sigma_1, \sigma_1-\sigma_0-d\sigma_{01}).
\]
Then we see that  $D \circ D =0$ so that we have a complex
$(A^*(\U),D)$. The $p$-th {\em \v{C}ech-de~Rham cohomology} of $\U$, denoted by
$H_D^p(\U)$, is the $p$-th cohomology of this complex. It is also abbreviated as  \v CdR cohomology.
We denote the class of a cocycle $\sigma$ by $[\sigma]$.
It can be shown that the map $A^p(M) \to A^p(\U)$ given by
$\o \mapsto (\o|_{U_0}, \o|_{U_1}, 0)$
induces an isomorphism
\begin{equationth}\label{dRCdR}
\a:H^p_d(M)\stackrel\sim\lra  H^p_D(\U).
\end{equationth}

\paragraph{Relative \v Cech-de~Rham cohomology\,:}
Let  $S$ be a closed set in $M$. Letting
$U_0=M\ssm  S$ and $U_1$ a \nbd\ of $S$ in $M$, we
consider the covering $\U=\{U_0,U_1\}$ of $M$.
 If we set
\[
A^p(\U,U_0)=\{\,\sigma\in A^p(\U)\mid \sigma_0=0\,\},
\]
we see that $(A^*(\U,U_0),D)$ is a subcomplex of $(A^*(\U),D)$. We denote by
$H^p_D(\U,U_0)$ the $p$-th cohomology of this complex. 

Suppose $S$ is a subcomplex relative to a $C^{\infty}$ triangulation $K_{0}$ of $M$.
Then we have
 a natural \iso\,:
\[
H^p_D(\U,U_0)\simeq H^p(M,M\ssm S;\C).
\]
Let $K$, $K'$ and $K^{*}$
be as before. We may assume that $U_{1}$ contains the open star $O_{K'}(S)$ of $S$ in $K'$. Let $R_{1}$ be an $m$-dimensional \mfd\ with piecewise $C^{\infty}$ boundary in $O_{K'}(S)$ containing $S$
in its interior, for example we may take the star $S_{K''}(S)$ of $S$ in the barycentric subdivision
$K''$ of $K'$ as $R_{1}$. We set $R_{01}=-\partial R_{1}$.

\begin{theorem}[\cite{Su1,Su2}] The Alexander \iso
\[
A:H^p_D(\U,U_0)\overset\sim\lra H_{m-p}(S,\C)
\]
is induced from the \homo
\[
A^{p}(\U,U_{0})\lra C_{m-p}^{K}(S,\C)\quad\text{given by}\ \sigma=(0,\sigma_{1},\sigma_{01})\mapsto\sum_{\bm{s}}\Big(\int_{\bm{s}^{*}\cap R_{1}}\sigma_{1}+\int_{\bm{s}^{*}\cap R_{01}}\sigma_{01}\Big)\bm{s},
\]
where $\bm{s}$ runs through the $(m-p)$-simplices of $K$ in $S$.\label{alexincdr} 
\end{theorem}
\subsubsection{Explicit expressions of  local classes}

If we use the \v{C}ech-de~Rham cohomology, the local class is given as follows. 
Let $W_0=W\ssm\vG_g$
and $W_1$ a \nbd\ of $\vG_g$ and consider the covering $\W=\{W_0,W_1\}$ of $W$. Let
$(0,\psi_1,\psi_{01})$ be a representative of the Thom class $\vP_{\vG_g}$ in $H^n_D(\W,W_0)\simeq
H^n(W,W\ssm\vG_g;\C)$. Let $U_{0}=M\ssm C$, $C=\op{Coin}(f,g)$, and $U_{1}$ a \nbd\ of $C$ \st\
$\tilde f(U_{1})\subset W_{1}$.  
Let $R_{1}$ and $R_{01}$ be as in Theorem \ref{alexincdr}.
Then from Theorem \ref{alexincdr}  and (\ref{coinlocthom}), we have\,:

\begin{theorem}\label{classloccdr} The class $\vL(f,g;C)$ in $H_{m-n}(C,\C)$ is represented by the cycle
\[
\sum_{\bm{s}}c_{\bm{s}}\bm{s},\qquad c_{\bm{s}}=\int_{\bm{s}^{*}\cap R_{1}}\tilde f^{*}\psi_{1}+\int_{\bm{s}^{*}\cap R_{01}}\tilde f^{*}\psi_{01},
\]
where $\bm{s}$ runs through the $(m-n)$-simplices of $K$ in $C$.
\end{theorem}

If $C$ has a finite number of connected components $(C_{\l})$, we take $U_{1}$ (and $K_{0}$) so that $U_{1}=\bigcup_{\l}U_{\l}$,
where  $U_{\l}\supset O_{K'}(C_{\l})$, for each $\l$,  and $U_{\l}\cap U_{\mu}=\emptyset$, if $\l\ne\mu$.
Then the class $\vL(f,g;C_{\l})$ in $H_{m-n}(C_{\l},\C)$ is represented by the above cycle with 
$\bm{s}$ running through the $(m-n)$-simplices of $K$ in $C_{\l}$.

\paragraph{(a) The case $m=n$ and $C_\l$ is compact\,:}
In this case, the local class $\vL(f,g;C_\l)$ is  in $H_0(C_\l)=\Z$ so that it is an integer.
By Theorem \ref{classloccdr} and the above remark, it is given by
\begin{equationth}\label{loccoin}
\vL(f,g;C_\l)=\int_{R_\l}\tilde f^*\psi_1+\int_{R_{0\l}}\tilde f^*\psi_{01},
\end{equationth}
where $R_{\l}=R_{1}\cap U_{\l}$ and $R_{0\l}=-\partial R_{\l}$.

Now suppose $C_\l=\{a\}$ is an isolated point.
A short  proof of the following formula using the  Thom class in the \v Cech-de~Rham cohomology is given in \cite{BBIS}\,:

\begin{theorem}\label{propisolated} We have\,{\rm :}
\[
\vL(f,g; a)=\op{deg}(g-f,a).
\]
\end{theorem}

From Theorems \ref{thlefgen}
and \ref{propisolated}, we recover Theorem \ref{lefcpf}.

\paragraph{(b) The case $C_\l$ is a pseudo-\mfd\,:}

\begin{definition}\label{defpmfd} A {\em pseudo-\mfd} $X$ of dimension $d$ in $M$ is a subcomplex of $M$ \wrt\
a triangulation of $M$ satisfying the following conditions\,:
\begin{enumerate}
\item[(1)] Every simplex in $X$ is a face of some $d$-simplex in $X$,
\item[(2)] Every $(d-1)$-simplex is the face of exactly two $d$-simplices,
\item[(3)] The $d$-simplices in $X$ can be oriented so that, if $\bm{s}$ is a $(d-1)$-simplex in $X$
and if $\bm{s}_1$ and $\bm{s}_2$ are the two simplices that contain $\bm{s}$ in their boundary, then 
the prescribed orientations of $\bm{s}_1$ and $\bm{s}_2$ induce opposite orientations of $\bm{s}$.
\end{enumerate}
\end{definition}

A pseudo-\mfd\ $X$ is said to be {\em oriented}, once   orientations of $d$-simplices in $X$ satisfying (3) above
are fixed. We say that $X$ is {\em \irr} if $X\ssm X^{d-2}$ is connected, where $X^{d-2}$ denotes the $(d-2)$-skeleton
of $X$. Then we have a decomposition into \irr\ components\,:
\[
X=\bigcup_iX_i.
\]
If $X$ is oriented, each $X_i$ carries a fundamental cycle, the union of $d$-simplices in $X_i$ and it defines
a class in $H_d(X)$. In fact $H_d(X)$ is generated by these classes.
\vv

Let $C_\l$ be a connected component of $C=\op{Coin}(f,g)$ as above and suppose it is an oriented 
pseudo-\mfd\ of dimension $m-n$. If $\bm{s}$ is an $(m-n)$-simplex in $C_\l$, ${\bm{s}^*}$ is an $n$-cell
\st\ ${\bm{s}^*}\cap C_\l=\{b_{\bm{s}}\}$, Thus $\deg (g-f)|_{\bm{s}^*}$ makes sense.

\begin{theorem}\label{pm} In the above situation, the local  class $\vL(f,g;C_\l)$ in $H_{m-n}(C_\l)$  is represented
 by the cycle
\[
\sum_{\bm{s}} \deg (g-f)|_{\bm{s}^*}\cdot \bm{s},
\]
where $\bm{s}$ runs through the $(m-n)$-simplices of $K$ in $C_\l$.
\end{theorem}

\begin{proof} This follows from (\ref{coinlocthom}) and Theorems \ref{classloccdr},  \ref{propisolated}.
\end{proof}

\begin{corollary}\label{corloccoin} In the above situation, if $C_\l=\bigcup_iC_{\l,i}$ is the \irr\ decomposition,
\[
\vL(f,g;C_\l)= \sum_i\deg (g_{\l,i}-f_{\l,i})\cdot C_{\l,i}\qquad\text{in}\ \ H_{m-n}(C_\l),
\]
where $g_{\l,i}-f_{\l,i}$ is the restriction of $g-f$ to a small ball of dimension $n$ transverse to $C_{\l,i}$ at a
non-singular point of $C_{\l,i}$.
\end{corollary}

In the case $C_\l$ is a \mfd, the above reduces to  Proposition 4.8 in \cite{BBIS}.
Note that we do not need the compactness of $M$ or $C_\l$.

\section{Coincidence of several maps}

\subsection{Global cohomology class}

Let $X$ be a topological space and  $N$  a compact connected oriented $n$-dimensional manifold.
In \cite{BLM} the authors consider $p$ maps $f_1, \dots, f_p : X \to N$, $p\ge 2$, and define a Lefschetz class 
$L(f_1, \ldots, f_p) \in H^{(p-1)n}(X,\Q)$. They  prove that $ L(f_1, \ldots, f_p)\ne 0$  implies that one has ``multi-coincidence"
$f_1(x) = \cdots = f_p(x) $ for some $x \in X$. 

We recall the definition of the  class 
$L(f_1, \ldots, f_p)$ and give explicit local contributions subsequently. Again, it is not necessary to consider the Thom class of $\vD_N$ or to assume the compactness of
$N$. Also, it is defined in $H^{(p-1)n}(X,\Z)$.

\begin{definition} The {\em Lefschetz coincidence cohomology class} of $f_1, \dots, f_p$ is defined as 
\[
L(f_1, \ldots, f_p) = L(f_1,f_2)\smallsmile\cdots\smallsmile L(f_{p-1},f_p)\qquad\text{in}\ \  H^{(p-1)n} (X).
\]
\end{definition}

\subsection{ Global homology classes}

Suppose $X=M$ is an oriented \mfd\ of dimension $m$.

\begin{definition}
We define the {\em global coincidence homology} class as
\[
\vL(f_1,\ldots ,f_p) =\vL(f_1,f_2)\Cdot\cdots\Cdot \vL(f_{p-1},f_p)\qquad\text{in}\ \ H_{m-(p-1)n}(M).
\]
\end{definition}

From Theorem \ref{thcoincoin}, we have\,:

\begin{theorem} We have the equality\,{\rm :}
\[
\vL(f_1,\ldots ,f_p) = P'_ML(f_1, \ldots, f_p).
\]
\end{theorem}

\subsection{Local homology classes}

Let us consider the case of $p$ maps $f_1, \dots, f_p : M \to N$. We set $W=M\times N$.
We denote by $C_{i,i+1}$  the coincidence set
\[
C_{i,i+1} = \op{Coin} (f_i,f_{i+1}) \qquad \text{ for } i=1, \cdots, p-1.
\]
Then we have 
\[
\bigcap_{i=1}^{p-1}C_{i,i+1} = \op{Coin} (f_1,\cdots,f_p),
\]
which we denote by $C$.

For each $i$, $1\le i \le p-1$, we have a commutative diagram:

\[
\SelectTips{cm}{}
\xymatrix 
{H^n(W, W\ssm \vG_{f_{i+1}}) \ar[r]^-{\sim}_-{A} \ar[d]^{{\tilde f}^*_i}
& H_m( \vG_{f_{i+1}}) \ar[d]^-{(M\cdot {}_{\tilde f_i}\;)_{C_{i,i+1}}}\\
H^n(M, M \ssm C_{i,i+1}) \ar[r]^-{\sim}_-{A} &H_{m-n}(C_{i,i+1})}
\]
and the class
\[
\vL(f_i,f_{i+1};C_{i,i+1})=(M\mCdot_{\tilde f_i}\hspace{.3mm}\vG_{f_{i+1}})_{C_{i,i+1}}\qquad\text{in}\ \ 
H_{m-n}(C_{i,i+1}).
\]

\begin{definition} We  define the {\em local coincidence class} as
\[
\vL(f_1,\ldots ,f_p ; C) = (\vL(f_1,f_2;C_{1,2})\Cdot\cdots\Cdot \vL(f_{p-1},f_p;C_{p-1,p}))_C\qquad
\text{in}\ \ H_{m-(p-1)n}(C).
\]
\end{definition}

Suppose $C = \op{Coin}(f_1,\dots,f_p)$ has a finite number of connected components $(C_\l)_\l$. Then we have 
 the local coincidence class 
$\vL(f_1,\dots,f_p;C_\l)$  in $H_{m-n}(C_\l)$ and we have a general coincidence point theorem
as in Section \ref{sectwomapsdd}\,:

\begin {theorem}\label{thcoinmult}
In the above situation, we have
\[
\vL(f_1,\dots,f_p) = \sum_\l (\iota_\l)_* \vL(f_1,\dots,f_p;C_\l) \qquad \text{ in}\ \  H_{m-(p-1)n}(M),
\]
where $\iota_\l:C_\l\hra M$ denotes the inclusion.
\end{theorem}
 
 From Corollary \ref{corloccoin}, we have an explicit expression of the local coincidence class.
 For simplicity we assume that each $C_{i,i+1}$ is an \irr\ oriented  pseudo-\mfd\ of dimension $m-n$.
 Let $B_{i,i+1}$ be a small ball of dimension $n$ in $M$ transverse to $C_{i,i+1}$ at a
non-singular point of $C_{i,i+1}$. Then we have

\begin{theorem}\label{locmult} In the above situation, we have the following formula in $H_{m-(p-1)n}(C)$\,{\rm :}
\[
\vL(f_1,\dots,f_p;C)=\deg ((f_2-f_1)|_{B_{1,2}})\cdots\deg ((f_p-f_{p-1})|_{B_{p-1,p}})\cdot 
(C_{1,2}\Cdot\cdots\Cdot C_{p-1,p})_C.
\]
\end{theorem}

If $m=(p-1)n$ and if the $C_{i,i+1}$'s intersect transversely, $C$ consists of isolated points and we have\,:
\begin{corollary}\label{corloccoingen} In the above situation, for a point $x$ in $C$,
\[
\vL(f_1,\dots,f_p;x)= \deg ((f_2-f_1)|_{B_{1,2}})\cdots\deg ((f_p-f_{p-1})|_{B_{p-1,p}}).
\]
\end{corollary}

\bigskip

J.-P. Brasselet

IML, CNRS, Campus de Luminy, Case 907

13288 Marseille Cedex 9, France

jean-paul.brasselet@univ-amu.fr

\bigskip

T. Suwa

Department of Mathematics

Hokkaido University

Sapporo 060-0810, Japan

tsuwa@sci.hokudai.ac.jp

\end{document}